\documentclass{amsart}
\usepackage{amsmath}
\usepackage{amsfonts}

\newtheorem{theorem}{Theorem}
\theoremstyle{plain}

\newtheorem{definition}{Definition}

\newtheorem{remark}{Remark}

\numberwithin{equation}{section}
\input{tcilatex}

\begin{document}
\title[A UNIFICATION OF MULTIPLE TWISTED EULER AND GENOCCHI POLYNOMIALS]{A
UNIFICATION OF THE MULTIPLE TWISTED EULER AND GENOCCHI NUMBERS AND
POLYNOMIALS ASSOCIATED WITH $p$-ADIC $q$-INTEGRAL ON $%
\mathbb{Z}
_{p}$ AT $q=-1$}
\author{Serkan Arac\i }
\address{University of Gaziantep, Faculty of Science and Arts, Department of
Mathematics, 27310 Gaziantep, TURKEY}
\email{mtsrkn@hotmail.com}
\author{Mehmet Acikgoz}
\address{University of Gaziantep, Faculty of Science and Arts, Department of
Mathematics, 27310 Gaziantep, TURKEY}
\email{acikgoz@gantep.edu.tr}
\author{Kyoung-Ho Park}
\address{Division of General Education-Mathematics, Kwangwoon University,
Seoul 139-171, Republic of Korea}
\email{sagamath@yahoo.co.kr}
\author{Hassan Jolany}
\address{School of Mathematics, Statistics and Computer Science, University
of Tehran, Iran }
\email{hassan.jolany@khayam.ut.ac.ir}
\date{October 26, 2011}

\begin{abstract}
The present paper deals with unification of the multiple twisted Euler and
Genocchi numbers and polynomials associated with $p$-adic $q$-integral on $%
\mathbb{Z}
_{p}$ at $q=-1$. Some earlier results of Ozden's papers in terms of
unification of the multiple twisted Euler and Genocchi numbers and
polynomials associated with $p$-adic $q$-integral on $%
\mathbb{Z}
_{p}$ at $q=-1$ can be deduced. We apply the method of generating function
and $p$-adic $q$-integral representation on $%
\mathbb{Z}
_{p}$, which are exploited to derive further classes of Euler polynomials
and Genocchi polynomials. To be more precise we summarize our results as
follows, we obtain some relations between H.Ozden's generating function and
fermionic $p$-adic $q$-integral on $%
\mathbb{Z}
_{p}$ at $q=-1$. Furthermore we derive Witt's type formula for the
unification of twisted Euler and Genocchi polynomials. Also we derive
distribution formula (Multiplication Theorem) for multiple twisted Euler and
Genocchi numbers and polynomials associated with $p$-adic $q$-integral on $%
\mathbb{Z}
_{p}$ at $q=-1$ which yields a deeper insight into the effectiveness of this
type of generalizations. Furthermore we define unification of multiple
twisted zeta function and we obtain an interpolation formula between
unification of multiple twisted zeta function and unification of the
multiple twisted Euler and Genocchi numbers at negative integer. Our new
generating function possess a number of interesting properties which we
state in this paper.
\end{abstract}

\maketitle

\section{Introduction, Definitions and Notations}

Bernoulli numbers introduced by Jacques Bernoulli (1654-1705), in the second
part of his treatise published in 1713, Ars conjectandi, at the time,
Bernoulli numbers were used for writing the infinite series expansions of
hyperbolic and trigonometric functions. Van den berg was the first to
discuss finding recurrence formulae for the Bernoulli numbers with arbitrary
sized gaps (1881). Ramanujan showed how gaps of size 7 could be found, and
explicitly wrote out the recursion for gaps, of size 6. Lehmer in 1934
extended these methods to Euler numbers, Genocchi numbers, and Lucas numbers
(1934), and calculated the 196-th Bernoulli numbers. The study of
generalized Bernoulli, Euler and Genocchi numbers and polynomials and their
combinatorial relations has received much attention \cite{B. N. Que}, \cite%
{M. S. Kim}, \cite{Luo 1}-\cite{Luo 4}, \cite{Liu}, \cite{Frappier}, \cite%
{Jolany 1}, \cite{Jolany 2}, \cite{Vandiver}. Generalized Bernoulli
polynomials, generalized Euler polynomials and generalized Genocchi numbers
and polynomials are the signs of very strong bond between elementary number
theory, complex analytic number theory, Homotopy theory (stable Homotopy
groups of spheres), differential topology (differential structures on
spheres), theory of modular forms (Eisenstein series), $p$-adic analytic
number theory ($p$-adic $L$-functions), quantum physics(quantum Groups). $p$%
-adic numbers were invented by Kurt Hensel around the end of the nineteenth
century. In spite of their being already one hundred years old, these
numbers are still today enveloped in an aura of mystery within the
scientific community. The $p$-adic integral was used in mathematical
physics, for instance, the functional equation of the $q$-zeta function, $q$%
-stirling numbers and $q$-Mahler theory of integration with respect to the
ring $%
\mathbb{Z}
_{p}$ together with Iwasawa's $p$-adic $q$-$L$ functions. Also the $p$-adic
interpolation functions of the Bernoulli and Euler polynomials have been
treated by Tsumura \cite{Tsumura} and Young \cite{Young}. Professor T.Kim 
\cite{kim 2}-\cite{Kim 24} also studied on $p$-adic interpolation functions
of these numbers and polynomials. In \cite{Carlitz}, Carlitz originally
constructed $q$-Bernoulli numbers and polynomials. These numbers and
polynomials are studied by many authors (see cf. \cite{kim 2}-\cite{Park}, 
\cite{Ozden}, \cite{Ozden 1}, \cite{Srivastava 1}). In the last decade, a
surprising number of papers appeared proposing new generalizations of the
Bernoulli, Euler and Genocchi polynomials to real and complex variables. In 
\cite{kim 2}-\cite{Kim 25}, Kim studied some families of multiple Bernoulli,
Euler and Genocchi numbers and polynomials. By using the fermionic $p$-adic
invariant integral on $%
\mathbb{Z}
_{p}$, he constructed $p$-adic Bernoulli, Euler and Genocchi numbers and
polynomials of higher order. A unification (and generalization) of Bernoulli
polynomials and Euler polynomials with a,b and c parameters first was
introduced and investigated by Q.-M.Luo \cite{Luo 1}, \cite{Luo 2}, \cite%
{Luo 3}, \cite{Luo 4}. After he with H.M.Srivastava defined unification (and
generalization) of Apostol type Bernoulli polynomials with a, b and c
parameters of higher order \cite{Luo 4}. After Hacer Ozden et al \cite{Ozden
1}. unified and extended the generating functions of the generalized
Bernoulli polynomials, the generalized Euler polynomials and the generalized
Genocchi polynomials associated with the positive real parameters a and b
and the complex parameter. Also they by applying the Mellin transformation
to the generating function of the unification of Bernoulli, Euler and
Genocchi polynomials, constructed a unification of the zeta functions.
Actually their definition provides a generalization and unification of the
Bernoulli, Euler and Genocchi polynomials and also of the
Apostol--Bernoulli, Apostol--Euler and Apostol--Genocchi polynomials, which
were considered in many earlier investigations by (among others) Srivastava
et al. \cite{Srivastava 2}, \cite{Srivastava 3}, \cite{Srivastava 4},
Karande \cite{Karande}. Also they by using a Dirichlet character defined
unification of the generating functions of the generalized Bernoulli, Euler
and Genocchi polynomials and numbers . T. Kim in \cite{Kim 20}, constructed
Apostol-Euler numbers and polynomials by using fermionic expression of $p$%
-adic $q$-integral at $q=-1$. In this paper by his method we derive several
properties for unification of the multiple twisted Euler and Genocchi
numbers and polynomials.

Let $p$ be a fixed odd prime number. Throughout this paper we use the
following notations, by $%
\mathbb{Z}
_{p}$ denotes the ring of $p$-adic rational integers, $%
\mathbb{Q}
$ denotes the field of rational numbers, $%
\mathbb{Q}
_{p}$ denotes the field of $p$-adic rational numbers, and $%
\mathbb{C}
_{p}$ denotes the completion of algebraic closure of $%
\mathbb{Q}
_{p}$. Let $%
\mathbb{N}
$ be the set of natural numbers and $%
\mathbb{N}
^{\ast }=%
\mathbb{N}
\cup \left\{ 0\right\} .$ The $p$-adic absolute value is defined by $%
\left\vert p\right\vert _{p}=\frac{1}{p}.$ In this paper we assume $%
\left\vert q-1\right\vert _{p}<1$ as an indeterminate. $\left[ x\right] _{q}$
is a $q$-extension of $x$ which is defined by $\left[ x\right] _{q}=\frac{%
1-q^{x}}{1-q}$,\ we note that $\lim_{q\rightarrow 1}\left[ x\right] _{q}=x$.

We say that $f$ is a uniformly differntiable function at a point $a\in 
\mathbb{Z}
_{p},$ if the difference quotient 
\begin{equation*}
F_{f}\left( x,y\right) =\frac{f\left( x\right) -f\left( y\right) }{x-y}
\end{equation*}

has a limit $f%
{\acute{}}%
\left( a\right) $ as $\left( x,y\right) \rightarrow \left( a,a\right) $ and
denote this by $f\in UD\left( 
\mathbb{Z}
_{p}\right) .$

\bigskip Let $UD\left( 
\mathbb{Z}
_{p}\right) $ be the set of uniformly differentiable function on $%
\mathbb{Z}
_{p}.$ For $f\in UD\left( 
\mathbb{Z}
_{p}\right) ,$ let us begin with the expressions%
\begin{equation*}
\frac{1}{\left[ p^{N}\right] }\sum_{0\leq x<p^{N}}f\left( x\right)
q^{x}=\sum_{0\leq x<p^{N}}f\left( x\right) \mu _{q}\left( x+p^{N}%
\mathbb{Z}
_{p}\right) ,
\end{equation*}

represents $p$-adic $q$-analogue of Riemann sums for $f.$ \ The integral of $%
f$ on $%
\mathbb{Z}
_{p}$ will be defined as the limit $\left( N\rightarrow \infty \right) $ of
these sums, when it exists. The $p$-adic $q$-integral of function $f\in
UD\left( 
\mathbb{Z}
_{p}\right) $ is defined by $Kim$ 
\begin{equation}
I_{q}\left( f\right) =\int_{%
\mathbb{Z}
_{p}}f\left( x\right) d\mu _{q}\left( x\right) =\lim_{N\rightarrow \infty }%
\frac{1}{\left[ p^{N}\right] _{q}}\sum_{x=0}^{p^{N}-1}f\left( x\right) q^{x}%
\text{ }  \label{equation 1}
\end{equation}

The bosonic integral is considered by $Kim$ as the bosonic limit $%
q\rightarrow 1,$ $I_{1}\left( f\right) =\lim_{q\rightarrow 1}I_{q}\left(
f\right) .$ Similarly, the fermionic $p$-adic integral on $%
\mathbb{Z}
_{p}$ is considered by $Kim$ as follows:%
\begin{equation*}
I_{-q}\left( f\right) =\lim_{q\rightarrow -q}I_{q}\left( f\right) =\int_{%
\mathbb{Z}
_{p}}f\left( x\right) d\mu _{-q}\left( x\right)
\end{equation*}

Assume that $q\rightarrow 1,$ then we have fermionic $p$-adic fermionic
integral on $%
\mathbb{Z}
_{p}$ as follows%
\begin{equation}
I_{-1}\left( f\right) =\lim_{q\rightarrow -1}I_{q}\left( f\right)
=\lim_{N\rightarrow \infty }\sum_{x=0}^{p^{N}-1}f\left( x\right) \left(
-1\right) ^{x}.  \label{equation 6}
\end{equation}

If we take $f_{1}\left( x\right) =f\left( x+1\right) $ in (\ref{equation 6}%
), then we have 
\begin{equation}
I_{-1}\left( f_{1}\right) +I_{-1}\left( f\right) =2f\left( 0\right) .
\label{equation 7}
\end{equation}

Let $p$ be a fixed prime. For a fixed positive integer $d$ with $\left(
p,d\right) =1,$ we set 
\begin{eqnarray*}
X &=&X_{d}=\lim_{\overset{\leftarrow }{N}}%
\mathbb{Z}
/dp^{N}%
\mathbb{Z}
, \\
X_{1} &=&%
\mathbb{Z}
_{p}, \\
X^{\ast } &=&\underset{\underset{\left( a,p\right) =1}{0<a<dp}}{\cup }a+dp%
\mathbb{Z}
_{p}
\end{eqnarray*}

and 
\begin{equation*}
a+dp^{N}%
\mathbb{Z}
_{p}=\left\{ x\in X\mid x\equiv a\left( \func{mod}dp^{N}\right) \right\} ,
\end{equation*}

where $a\in 
\mathbb{Z}
$ satisfies the condition $0\leq a<dp^{N}.$

\begin{definition}
(see, cf. \cite{ozden 2}). A unification $y_{n,\beta }\left( x:k,a,b\right) $
of the Bernoulli, Euler and Genochhi polynomials is given by the following
generating function:%
\begin{eqnarray}
F_{a,b}\left( x;t;k,\beta \right) &=&\frac{2\left( \frac{t}{2}\right) ^{k}}{%
\beta ^{b}e^{t}-a^{b}}e^{xt}=\sum_{n=0}^{\infty }y_{n,\beta }\left(
x:k,a,b\right) \frac{t^{n}}{n!}\text{ \ }\left( \left\vert t+\log \left( 
\frac{\beta }{a}\right) \right\vert <2\pi ;\text{ }x\in 
\mathbb{R}
\right)  \notag \\
&&\left( k\in 
\mathbb{N}
^{\ast };a,b\in 
\mathbb{R}
^{+};\beta \in 
\mathbb{C}
\right) ,  \label{equation 8}
\end{eqnarray}%
where as usual $%
\mathbb{R}
^{+},$ and $%
\mathbb{C}
$ denote the sets of positive real numbers and complex numbers,
respectively, $%
\mathbb{R}
$ being the set of real numbers.
\end{definition}

Observe that, if we put $x=0$ in the generating function (\ref{equation 8}),
then we obtain the corresponding unification of the generating functions of
Bernoulli, Euler and Genocchi numbers. Then we have%
\begin{equation*}
y_{n,\beta }\left( 0:k,a,b\right) =y_{n,\beta }\left( k,a,b\right) .
\end{equation*}

We are now ready to give relationship between the Ozden's generating
function and the fermionic $p$-adic $q$-integral on $%
\mathbb{Z}
_{p}$ at $q=-1$ with the following theorem:

\begin{theorem}
The following relationship holds:%
\begin{equation}
a^{-b}\left( \frac{t}{2}\right) ^{k}\int_{%
\mathbb{Z}
_{p}}\left( -1\right) ^{x+1}\left( \frac{\beta }{a}\right) ^{bx}e^{tx}d\mu
_{-1}\left( x\right) =\sum_{n=0}^{\infty }y_{n,\beta }\left( k,a,b\right) 
\frac{t^{n}}{n!}.  \label{equation 9}
\end{equation}
\end{theorem}

\begin{proof}
We set $f\left( x\right) =a^{-b}\left( \frac{t}{2}\right) ^{k}\left(
-1\right) ^{x+1}\left( \frac{\beta }{a}\right) ^{bx}e^{tx}$ in (\ref%
{equation 7}), it is easy to show 
\begin{eqnarray*}
a^{-b}\left( \frac{t}{2}\right) ^{k}\left( -\left( \frac{\beta }{a}\right)
^{b}e^{t}+1\right) \int_{%
\mathbb{Z}
_{p}}\left( -1\right) ^{x+1}\left( \frac{\beta }{a}\right) ^{bx}e^{tx}d\mu
_{-1}\left( x\right) &=&-\frac{2\left( \frac{t}{2}\right) ^{k}}{a^{b}} \\
a^{-b}\left( \frac{t}{2}\right) ^{k}\int_{%
\mathbb{Z}
_{p}}\left( -1\right) ^{x+1}\left( \frac{\beta }{a}\right) ^{bx}e^{tx}d\mu
_{-1}\left( x\right) &=&\frac{2\left( \frac{t}{2}\right) ^{k}}{\beta
^{b}e^{t}-a^{b}}
\end{eqnarray*}

So, we complete the proof.
\end{proof}

\begin{theorem}
Then the following identity holds:%
\begin{equation*}
\int_{%
\mathbb{Z}
_{p}}\left( -1\right) ^{x+1}\left( \frac{\beta }{a}\right) ^{bx}x^{n-k}d\mu
_{-1}\left( x\right) =2^{k}a^{b}\frac{\left( n-k\right) !}{n!}y_{n,\beta
}\left( k,a,b\right) .
\end{equation*}
\end{theorem}

\begin{proof}
From (\ref{equation 9}) and by using the taylor expansion of $e^{tx},$ we
readily see that,%
\begin{equation*}
\sum_{n=0}^{\infty }\left( 2^{-k}a^{-b}\int_{%
\mathbb{Z}
_{p}}\left( -1\right) ^{x+1}\left( \frac{\beta }{a}\right) ^{bx}x^{n}d\mu
_{-1}\left( x\right) \right) \frac{t^{n+k}}{n!}=\sum_{n=0}^{\infty
}y_{n,\beta }\left( k,a,b\right) \frac{t^{n}}{n!}
\end{equation*}

By comparing coefficients of $t^{n}$ in the both sides of the above
equation, we arrive at the desired result.
\end{proof}

Similarly, we obtain te following theorem for a unification of the Euler and
Genocchi polynomials as follows:

\begin{theorem}
Then the following identity holds:%
\begin{equation}
\int_{%
\mathbb{Z}
_{p}}\left( -1\right) ^{y+1}\left( \frac{\beta }{a}\right) ^{by}\left(
x+y\right) ^{n}d\mu _{-1}\left( y\right) =2^{k}a^{b}\frac{n!}{\left(
n+k\right) !}y_{n+k,\beta }\left( x:k,a,b\right) .  \label{equation 16}
\end{equation}
\end{theorem}

From the binomial theorem in (\ref{equation 16}), we possess the following
theorem:

\begin{theorem}
The following relation holds:%
\begin{equation*}
\frac{y_{n+k,\beta }\left( x:k,a,b\right) }{\binom{n+k}{k}}=\sum_{m=0}^{n}%
\frac{\binom{n}{m}}{\binom{m+k}{k}}y_{m+k,\beta }\left( k,a,b\right) x^{n-m}
\end{equation*}
\end{theorem}

\begin{proof}
By using (\ref{equation 16}) and binomial theorem, we express the following
relation%
\begin{equation*}
\sum_{m=0}^{n}\binom{n}{m}\left( \int_{%
\mathbb{Z}
_{p}}\left( -1\right) ^{y+1}\left( \frac{\beta }{a}\right) ^{by}y^{m}d\mu
_{-1}\left( y\right) \right) x^{n-m}=2^{k}a^{b}\frac{n!}{\left( n+k\right) !}%
y_{n+k,\beta }\left( x:k,a,b\right)
\end{equation*}

By using $p$-adic $q$-integral on $%
\mathbb{Z}
_{p}$ at $q=-1,$ we arrive at the desired proof of the theorem.
\end{proof}

Now, we consider symmetric properties of this type of polynomials as follows:

\begin{theorem}
The following relation holds:%
\begin{equation*}
y_{n,\beta ^{-1}}\left( 1-x:k,a^{-1},b\right) =\left( -1\right)
^{k+n+1}\beta ^{b}a^{b}y_{n,\beta }\left( x:k,a,b\right) .
\end{equation*}
\end{theorem}

\begin{proof}
We set $x\rightarrow 1-x,$ $\beta \rightarrow \beta ^{-1}$ and $a\rightarrow
a^{-1}$ into (\ref{equation 16}), that is%
\begin{eqnarray*}
&&\int_{%
\mathbb{Z}
_{p}}\left( -1\right) ^{y+1}\left( \frac{\beta ^{-1}}{a^{-1}}\right)
^{by}\left( 1-x+y\right) ^{n}d\mu _{-1}\left( y\right) \\
&=&\left( -1\right) ^{n}\int_{%
\mathbb{Z}
_{p}}\left( -1\right) ^{y+1}\left( \frac{\beta }{a}\right) ^{-by}\left(
x-1+y\right) ^{n}d\mu _{-1}\left( y\right) \\
&=&\left( -1\right) ^{k+n+1}\beta ^{b}a^{b}y_{n,\beta }\left( x:k,a,b\right)
\end{eqnarray*}

Thus, we complete proof of the theorem.
\end{proof}

Ozden has obtained distribution formula for $y_{n,\beta }\left(
x:k,a,b\right) .$ We will also obtain distribution formula by using $p$-adic 
$q$-integral on $%
\mathbb{Z}
_{p}$ at $q=-1.$

\begin{theorem}
The following identity holds:%
\begin{equation*}
y_{n,\beta }\left( x:k,a,b\right) =a^{b\left( d-1\right)
}d^{n-k}\sum_{j=0}^{d-1}\left( \frac{\beta }{a}\right) ^{bj}y_{n,\beta
^{d}}\left( \frac{x+j}{d}:k,a^{d},b\right) .
\end{equation*}
\end{theorem}

\begin{proof}
By using definition of the $p$-adic integral on $%
\mathbb{Z}
_{p},$ we compute%
\begin{eqnarray*}
2^{k}a^{b}\frac{n!}{\left( n+k\right) !}y_{n+k,\beta }\left( x:k,a,b\right)
&=&\int_{%
\mathbb{Z}
_{p}}\left( -1\right) ^{y+1}\left( \frac{\beta }{a}\right) ^{by}\left(
x+y\right) ^{n}d\mu _{-1}\left( y\right) \\
&=&\lim_{N\rightarrow \infty }\sum_{y=0}^{dp^{N}-1}\left( -1\right)
^{y+1}\left( \frac{\beta }{a}\right) ^{by}\left( x+y\right) ^{n}\left(
-1\right) ^{y} \\
&=&d^{n}\sum_{j=0}^{d-1}\left( \frac{\beta }{a}\right)
^{bj}\lim_{N\rightarrow \infty }\sum_{y=0}^{p^{N}-1}\left( -1\right)
^{y+1}\left( \frac{\beta }{a}\right) ^{bdy}\left( \frac{x+j}{d}+y\right)
^{n}\left( -1\right) ^{y} \\
&=&d^{n}\sum_{j=0}^{d-1}\left( \frac{\beta }{a}\right) ^{bj}\int_{%
\mathbb{Z}
_{p}}\left( -1\right) ^{y+1}\left( \frac{\beta ^{d}}{a^{d}}\right)
^{by}\left( \frac{x+j}{d}+y\right) ^{n}d\mu _{-1}\left( y\right) \\
&=&d^{n}\sum_{j=0}^{d-1}\left( \frac{\beta }{a}\right) ^{bj}2^{k}a^{db}\frac{%
n!}{\left( n+k\right) !}y_{n+k,\beta ^{d}}\left( \frac{x+j}{d}%
:k,a^{d},b\right)
\end{eqnarray*}

Substituting $n$ by $n-k$, we will be completed the proof of theorem.
\end{proof}

\begin{remark}
This distribution for $y_{n,\beta }\left( x:k,a,b\right) $ is also
introduced by Ozden cf.\cite{ozden 2}.
\end{remark}

\begin{definition}
(see, for detail \cite{Ozden 1})Let $\chi $ be a Dirichlet character with
conductor $d\in 
\mathbb{N}
.$ The generating functions of the generalized Bernoulli, Euler and Genocchi
polynomials with parameters $a,$ $b,$ $\beta $ and $k$ have been defined by
Ozden, Simsek and Srivastava as follows:{}%
\begin{eqnarray}
&&\tciFourier _{\chi ,\beta }\left( t,k,a,b\right)  \label{equation 17} \\
&=&2\left( \frac{t}{2}\right) ^{k}\sum_{j=1}^{d}\frac{\chi \left( j\right)
\left( \frac{\beta }{a}\right) ^{j}e^{jt}}{\beta ^{bd}e^{dt}-a^{bd}}  \notag
\\
&=&\sum_{n=0}^{\infty }y_{n,\chi ,\beta }\left( x:k,a,b\right) \frac{t^{n}}{%
n!},\text{ }\left( \left\vert t+b\log \left( \frac{\beta }{a}\right)
\right\vert <2\pi ;\text{ }d,k\in 
\mathbb{N}
;\text{ }a,b\in 
\mathbb{R}
^{+};\text{ }\beta \in 
\mathbb{C}
\right)  \notag
\end{eqnarray}
\end{definition}

\hspace*{0pt}By using $p$-adic integral on $%
\mathbb{Z}
_{p},$ we can obtain (\ref{equation 17}){} with the following theorem:

\begin{theorem}
Let $\chi $ be a Dirichlet's character with conductor $d\in 
\mathbb{N}
.$ Then the following relation holds%
\begin{equation}
a^{b\left( 1-d\right) }\left( \frac{t}{2}\right) ^{k}\int_{%
\mathbb{Z}
_{p}}\chi \left( x\right) \left( -1\right) ^{x+1}\left( \frac{\beta }{a}%
\right) ^{bx}e^{tx}d\mu _{-1}\left( x\right) =2^{1-k}t^{k}\sum_{j=1}^{d}%
\frac{\chi \left( j\right) \left( \frac{\beta }{a}\right) ^{bj}e^{tj}}{\beta
^{db}e^{dt}-a^{db}}  \label{equation 32}
\end{equation}

\begin{proof}
From the definition of $p$-adic $q$-integral on $%
\mathbb{Z}
_{p}$ at $q=-1,$ we compute%
\begin{eqnarray*}
&&a^{b\left( 1-d\right) }\left( \frac{t}{2}\right) ^{k}\int_{%
\mathbb{Z}
_{p}}\chi \left( x\right) \left( -1\right) ^{x+1}\left( \frac{\beta }{a}%
\right) ^{bx}e^{tx}d\mu _{-1}\left( x\right) \\
&=&a^{b\left( 1-d\right) }\left( \frac{t}{2}\right) ^{k}\lim_{N\rightarrow
\infty }\sum_{x=0}^{dp^{N}-1}\chi \left( x\right) \left( -1\right)
^{x+1}\left( \frac{\beta }{a}\right) ^{bx}e^{tx}\left( -1\right) ^{x} \\
&=&\frac{1}{d^{k}}\sum_{j=1}^{d}\chi \left( j\right) \left( \frac{\beta }{a}%
\right) ^{bj}e^{tj}\left( \frac{1}{a^{db}}\left( \frac{td}{2}\right)
^{k}\lim_{N\rightarrow \infty }\sum_{x=0}^{p^{N}-1}\left( -1\right)
^{x+1}\left( \frac{\beta ^{d}}{a^{d}}\right) ^{bx}e^{tdx}\left( -1\right)
^{x}\right) \\
&=&\frac{1}{d^{k}}\sum_{j=1}^{d}\chi \left( j\right) \left( \frac{\beta }{a}%
\right) ^{bj}e^{tj}\left( \frac{2\left( \frac{td}{2}\right) ^{k}}{\beta
^{db}e^{dt}-a^{db}}\right) \\
&=&2^{1-k}t^{k}\sum_{j=1}^{d}\frac{\chi \left( j\right) \left( \frac{\beta }{%
a}\right) ^{bj}e^{tj}}{\beta ^{db}e^{dt}-a^{db}}
\end{eqnarray*}%
Thus, we arrive at the desired result.
\end{proof}
\end{theorem}

{}\bigskip By expression of (\ref{equation 32}), we get the following
equation{}{}%
\begin{equation}
a^{b\left( 1-d\right) }\left( \frac{t}{2}\right) ^{k}\int_{%
\mathbb{Z}
_{p}}\chi \left( x\right) \left( -1\right) ^{x+1}\left( \frac{\beta }{a}%
\right) ^{bx}e^{tx}d\mu _{-1}\left( x\right) =\sum_{n=0}^{\infty }y_{n,\chi
,\beta }\left( x:k,a,b\right) \frac{t^{n}}{n!}.  \label{equation 30}
\end{equation}

We are now ready to give distribution formula for generalized Euler and
Genocchi polynomials by using $p$-adic $q$-integral on $%
\mathbb{Z}
_{p}$ at $q=-1$ by means of theorem.

\begin{theorem}
For any $n,k,d\in 
\mathbb{N}
$ $a,b\in 
\mathbb{R}
^{+};$ $\beta \in 
\mathbb{C}
,$ we have%
\begin{equation*}
y_{n,\chi ,\beta }\left( x:k,a,b\right) =d^{n-k}\sum_{j=0}^{d-1}\chi \left(
j\right) \left( \frac{\beta }{a}\right) ^{bj}y_{n,\beta ^{d}}\left( \frac{x+j%
}{d}:k,a^{d},b\right) .
\end{equation*}

\begin{proof}
By expression of (\ref{equation 30}), we compute as follows:%
\begin{eqnarray*}
&&\sum_{n=0}^{\infty }y_{n,\chi ,\beta }\left( x:k,a,b\right) \frac{t^{n}}{n!%
} \\
&=&a^{b\left( 1-d\right) }\left( \frac{t}{2}\right) ^{k}\int_{%
\mathbb{Z}
_{p}}\chi \left( y\right) \left( -1\right) ^{y+1}\left( \frac{\beta }{a}%
\right) ^{by}e^{t\left( x+y\right) }d\mu _{-1}\left( y\right) \\
&=&a^{b\left( 1-d\right) }\left( \frac{t}{2}\right) ^{k}\lim_{N\rightarrow
\infty }\sum_{y=0}^{dp^{N}-1}\chi \left( y\right) \left( -1\right)
^{y+1}\left( \frac{\beta }{a}\right) ^{by}e^{t\left( x+y\right) }\left(
-1\right) ^{y} \\
&=&\frac{1}{d^{k}}\sum_{j=0}^{d-1}\chi \left( j\right) \left( \frac{\beta }{a%
}\right) ^{bj}\left( \frac{1}{a^{db}}\left( \frac{dt}{2}\right)
^{k}\lim_{N\rightarrow \infty }\sum_{y=0}^{p^{N}-1}\left( -1\right)
^{y+1}\left( \frac{\beta ^{d}}{a^{d}}\right) ^{by}e^{dt\left( \frac{x+j}{d}%
+y\right) }\left( -1\right) ^{y}\right) \\
&=&\frac{1}{d^{k}}\sum_{j=0}^{d-1}\chi \left( j\right) \left( \frac{\beta }{a%
}\right) ^{bj}\left( \frac{1}{a^{db}}\left( \frac{dt}{2}\right) ^{k}\int_{%
\mathbb{Z}
_{p}}\left( -1\right) ^{y+1}\left( \frac{\beta ^{d}}{a^{d}}\right)
^{by}e^{dt\left( \frac{x+j}{d}+y\right) }d\mu _{-1}\left( y\right) \right) \\
&=&\frac{1}{d^{k}}\sum_{j=0}^{d-1}\chi \left( j\right) \left( \frac{\beta }{a%
}\right) ^{bj}\left( \sum_{n=0}^{\infty }d^{n}y_{n,\beta ^{d}}\left( \frac{%
x+j}{d}:k,a^{d},b\right) \frac{t^{n}}{n!}\right) \\
&=&\sum_{n=0}^{\infty }\left( d^{n-k}\sum_{j=0}^{d-1}\chi \left( j\right)
\left( \frac{\beta }{a}\right) ^{bj}y_{n,\beta ^{d}}\left( \frac{x+j}{d}%
:k,a^{d},b\right) \right) \frac{t^{n}}{n!}.
\end{eqnarray*}%
So, we complete the proof of theorem.
\end{proof}
\end{theorem}

\section{\protect\bigskip {}New properties on the unification of multiple
twisted Euler and Genocchi polynomials}

In this section, we introduce a unification of the twisted Euler and
Genocchi polynomials. We assume that $q\in 
\mathbb{C}
_{p}$ with $\left\vert 1-q\right\vert _{p}<1.$ For $n\in 
\mathbb{N}
,$ by the definition of the $p$-adic integral on $%
\mathbb{Z}
_{p},$ we have%
\begin{equation}
I_{-1}\left( f_{n}\right) +\left( -1\right) ^{n-1}I_{-1}\left( f\right)
=2\sum_{x=0}^{n-1}f\left( x\right) \left( -1\right) ^{n-1-x}
\label{equation 18}
\end{equation}

where $f_{n}\left( x\right) =f\left( x+n\right) .$

Let $T_{p}=\underset{n\geq 1}{\cup }C_{p^{n}}=\lim_{n\rightarrow \infty
}C_{p^{n}}=C_{p^{\infty }}$ be the locally constant space, where $%
C_{p^{n}}=\left\{ w\mid w^{p^{n}}=1\right\} $ is the cylic group of order $%
p^{n}.$ For $w\in T_{p},$ we denote the locally constant function by 
\begin{equation}
\phi _{w}:%
\mathbb{Z}
_{p}\rightarrow 
\mathbb{C}
_{p},\text{ }x\rightarrow w^{x},\text{ }  \label{equation 19}
\end{equation}

If we set $f\left( x\right) =\phi _{w}\left( x\right) a^{-b}\left( \frac{t}{2%
}\right) ^{k}\left( -1\right) ^{x+1}\left( \frac{\beta }{a}\right)
^{bx}e^{tx},$ then we have 
\begin{equation}
a^{-b}\left( \frac{t}{2}\right) ^{k}\int_{%
\mathbb{Z}
_{p}}\phi _{w}\left( x\right) \left( -1\right) ^{x+1}\left( \frac{\beta }{a}%
\right) ^{bx}e^{tx}d\mu _{-1}\left( x\right) =\frac{2\left( \frac{t}{2}%
\right) ^{k}}{w\beta ^{b}e^{t}-a^{b}}  \label{equation 20}
\end{equation}

We now define unification of twisted Euler and Genocchi polynomials as
follows:%
\begin{equation*}
\frac{2\left( \frac{t}{2}\right) ^{k}}{w\beta ^{b}e^{t}-a^{b}}%
=\sum_{n=0}^{\infty }y_{n,w,\beta }\left( k,a,b\right) \frac{t^{n}}{n!},
\end{equation*}

We note that by substituting $w=1,$ we obtain Ozden's generating function (%
\ref{equation 8}). From (\ref{equation 19}) and (\ref{equation 20}), we
obtain witt's type formula for a unificaton of twisted Euler and Genocchi
polynomials as follows:%
\begin{equation}
a^{-b}2^{-k}\int_{%
\mathbb{Z}
_{p}}\phi _{w}\left( x\right) \left( -1\right) ^{x+1}\left( \frac{\beta }{a}%
\right) ^{bx}x^{n}d\mu _{-1}\left( x\right) =\frac{y_{n+k,w,\beta }\left(
k,a,b\right) }{k!\binom{n+k}{k}}  \label{equation 21}
\end{equation}

for each $w\in T_{p}$ and $n\in 
\mathbb{N}
.$

We now establish Witt's type formula for the unification of multiple twisted
Euler and Genocchi polynomials by the following theorem.

\begin{definition}
Let be $w\in T_{p},$ $n,h,k\in 
\mathbb{N}
$ $a,b\in 
\mathbb{R}
^{+};$ $\beta \in 
\mathbb{C}
,$ we define%
\begin{eqnarray}
&&a^{-hb}2^{-hk}\underset{h-times}{\underbrace{\int_{%
\mathbb{Z}
_{p}}...\int_{%
\mathbb{Z}
_{p}}}}\phi _{w}\left( x_{1}+...+x_{h}\right) \left( -1\right)
^{x_{1}+...+x_{h}+h}  \label{equation 22} \\
&&\times \left( \frac{\beta }{a}\right) ^{b\left( x_{1}+...+x_{h}\right)
}\left( x_{1}+...+x_{h}\right) ^{n}d\mu _{-1}\left( x_{1}\right) ...d\mu
_{-1}\left( x_{h}\right)  \notag \\
&=&\frac{y_{n+kh,w,\beta }^{\left( h\right) }\left( k,a,b\right) }{\left(
kh\right) !\binom{n+kh}{kh}}.  \notag
\end{eqnarray}
\end{definition}

\begin{remark}
Taking $h=1$ into (\ref{equation 22}), we get the unification of the twisted
Euler and Genocchi polynomials $y_{n,w,\beta }\left( k,a,b\right) .$
\end{remark}

\begin{remark}
By substituting $h=1$ and $w=1,$ we obtain a special case of the unification
of Euler and Genocchi polynomials $y_{n,\beta }\left( k,a,b\right) .$
\end{remark}

\begin{theorem}
For any $w\in T_{p},$ $n,h,k\in 
\mathbb{N}
$ $a,b\in 
\mathbb{R}
^{+};$ $\beta \in 
\mathbb{C}
,$ 
\begin{equation*}
\frac{y_{n+kh,w,\beta }^{\left( h\right) }\left( k,a,b\right) }{\left(
kh\right) !\binom{n+kh}{kh}}=\sum_{\underset{l_{1},...,l_{h}\geq 0}{%
l_{1}+...+l_{h}=n}}\frac{n!}{l_{1}!...l_{h}!}\dprod\limits_{i=1}^{h}\frac{%
y_{l_{i}+kh,w,\beta }^{\left( h\right) }\left( k,a,b\right) }{\left(
kh\right) !\binom{l_{i}+kh}{kh}}
\end{equation*}

\begin{proof}
By using definition of the multiple twisted a unification of Euler and
Genocchi numbers and polynomials, and, definition of $\left(
x_{1}+x_{2}+...+x_{h}\right) ^{n}=\sum_{\underset{l_{1},...,l_{h}\geq 0}{%
l_{1}+...+l_{h}=n}}\frac{n!}{l_{1}!...l_{h}!}%
x_{1}^{l_{1}}x_{2}^{l_{2}}...x_{h}^{l_{h}},$ we see that,%
\begin{eqnarray*}
&&a^{-hb}2^{-hk}\underset{h-times}{\underbrace{\int_{%
\mathbb{Z}
_{p}}...\int_{%
\mathbb{Z}
_{p}}}}\phi _{w}\left( x_{1}+...+x_{h}\right) \left( -1\right)
^{x_{1}+...+x_{h}+h}\left( \frac{\beta }{a}\right) ^{b\left(
x_{1}+...+x_{h}\right) } \\
&&\times \left( x+x_{1}+...+x_{h}\right) ^{n}d\mu _{-1}\left( x_{1}\right)
...d\mu _{-1}\left( x_{h}\right) \\
&=&\sum_{\underset{l_{1},...,l_{h}\geq 0}{l_{1}+...+l_{h}=n}}\frac{n!}{%
l_{1}!...l_{h}!}\left( a^{-b}2^{-k}\int_{%
\mathbb{Z}
_{p}}w^{x_{1}}\left( \frac{\beta }{a}\right) ^{bx_{1}}x_{1}^{l_{1}}d\mu
_{-1}\left( x_{1}\right) \right) \times \\
&&...\times \left( a^{-b}2^{-k}\int_{%
\mathbb{Z}
_{p}}w^{x_{h}}\left( \frac{\beta }{a}\right) ^{bx_{h}}x_{h}^{l_{h}}d\mu
_{-1}\left( x_{h}\right) \right) \\
&=&\sum_{\underset{l_{1},...,l_{h}\geq 0}{l_{1}+...+l_{h}=n}}\frac{n!}{%
l_{1}!...l_{h}!}\dprod\limits_{j=1}^{h}\frac{y_{l_{i}+kh,w,\beta }^{\left(
h\right) }\left( k,a,b\right) }{\left( kh\right) !\binom{l_{i}+kh}{kh}}
\end{eqnarray*}%
Thus, we arrive at the desired result.
\end{proof}
\end{theorem}

From these formulas, we can define the unification of the twisted Euler and
Genocchi polynomials as follows:%
\begin{equation}
\left( \frac{2\left( \frac{t}{2}\right) ^{k}}{w\beta ^{b}e^{t}-a^{b}}\right)
^{h}e^{xt}=\sum_{n=0}^{\infty }y_{n,w,\beta }^{\left( h\right) }\left(
x:k,a,b\right) \frac{t^{n}}{n!},  \label{equation 24}
\end{equation}

So from above, we get the Witt's type formula for $y_{n,w,\beta }^{\left(
h\right) }\left( x:k,a,b\right) $ as follows.

\begin{theorem}
For any $w\in T_{p},$ $n,h,k\in 
\mathbb{N}
$ $a,b\in 
\mathbb{R}
^{+};$ $\beta \in 
\mathbb{C}
,$ we get%
\begin{eqnarray}
&&a^{-hb}2^{-hk}\underset{h-times}{\underbrace{\int_{%
\mathbb{Z}
_{p}}...\int_{%
\mathbb{Z}
_{p}}}}\phi _{w}\left( x_{1}+...+x_{h}\right) \left( -1\right)
^{x_{1}+...+x_{h}+h}\left( \frac{\beta }{a}\right) ^{b\left(
x_{1}+...+x_{h}\right) }  \notag \\
&&\times \left( x+x_{1}+...+x_{h}\right) ^{n}d\mu _{-1}\left( x_{1}\right)
...d\mu _{-1}\left( x_{h}\right)  \notag \\
&=&\frac{y_{n+kh,w,\beta }^{\left( h\right) }\left( x:k,a,b\right) }{\left(
kh\right) !\binom{n+kh}{kh}}  \label{equation 27}
\end{eqnarray}
\end{theorem}

Note that%
\begin{equation}
\left( x+x_{1}+x_{2}+...+x_{h}\right) ^{n}=\sum_{\underset{%
l_{1},...,l_{h}\geq 0}{l_{1}+...+l_{h}=n}}\frac{n!}{l_{1}!...l_{h}!}%
x_{1}^{l_{1}}x_{2}^{l_{2}}...\left( x+x_{h}\right) ^{l_{h}}
\label{equation 26}
\end{equation}

We obtain the sum of powers of consecutive a unification of multiple twisted
Euler and Genocchi polynomials as follows:

\begin{theorem}
For any $w\in T_{p},$ $n,h,k\in 
\mathbb{N}
$ $a,b\in 
\mathbb{R}
^{+};$ $\beta \in 
\mathbb{C}
,$ we get%
\begin{equation*}
\frac{y_{n+kh,w,\beta }^{\left( h\right) }\left( x:k,a,b\right) }{\left(
kh\right) !\binom{n+kh}{kh}}=\sum_{\underset{l_{1},...,l_{h}\geq 0}{%
l_{1}+...+l_{h}=n}}\frac{n!}{l_{1}!...l_{h}!}\frac{y_{l_{h}+kh,w,\beta
}^{\left( h\right) }\left( x:k,a,b\right) }{\left( kh\right) !\binom{l_{h}+kh%
}{kh}}\dprod\limits_{j=1}^{h-1}\frac{y_{l_{i}+kh,w,\beta }^{\left( h\right)
}\left( k,a,b\right) }{\left( kh\right) !\binom{l_{i}+kh}{kh}}.
\end{equation*}
\end{theorem}

\begin{proof}
By (\ref{equation 27}) and (\ref{equation 26}), we see that,%
\begin{eqnarray*}
&&a^{-hb}2^{-hk}\underset{h-times}{\underbrace{\int_{%
\mathbb{Z}
_{p}}...\int_{%
\mathbb{Z}
_{p}}}}\phi _{w}\left( x_{1}+...+x_{h}\right) \left( -1\right)
^{x_{1}+...+x_{h}+h}\left( \frac{\beta }{a}\right) ^{b\left(
x_{1}+...+x_{h}\right) } \\
&&\times \left( x+x_{1}+...+x_{h}\right) ^{n}d\mu _{-1}\left( x_{1}\right)
...d\mu _{-1}\left( x_{h}\right) \\
&=&\sum_{\underset{l_{1},...,l_{h}\geq 0}{l_{1}+...+l_{h}=n}}\frac{n!}{%
l_{1}!...l_{h}!}\left( a^{-b}2^{-k}\int_{%
\mathbb{Z}
_{p}}w^{x_{1}}\left( \frac{\beta }{a}\right) ^{bx_{1}}x_{1}^{l_{1}}d\mu
_{-1}\left( x_{1}\right) \right) \times \\
&&...\times \left( a^{-b}2^{-k}\int_{%
\mathbb{Z}
_{p}}w^{x_{h}}\left( \frac{\beta }{a}\right) ^{bx_{h}}\left( x+x_{h}\right)
^{l_{h}}d\mu _{-1}\left( x_{h}\right) \right) \\
&=&\sum_{\underset{l_{1},...,l_{h}\geq 0}{l_{1}+...+l_{h}=n}}\frac{n!}{%
l_{1}!...l_{h}!}\frac{y_{l_{h}+kh,w,\beta }^{\left( h\right) }\left(
x:k,a,b\right) }{\left( kh\right) !\binom{l_{h}+kh}{kh}}\dprod%
\limits_{j=1}^{h-1}\frac{y_{l_{i}+kh,w,\beta }^{\left( h\right) }\left(
k,a,b\right) }{\left( kh\right) !\binom{l_{i}+kh}{kh}}
\end{eqnarray*}

So, we complete the proof of the theorem.
\end{proof}

\section{A unification of multiple twisted Zeta functions}

Our goal in this section is to establish a unification of multiple twisted
zeta functions which interpolates of a unification of multiple twisted Euler
and Genocchi polynomials at negative integers. For $q\in 
\mathbb{C}
,$ $\left\vert q\right\vert <1$ and $w\in T_{p},$ a unification of multiple
twisted Euler and Genocchi polynomials are considered as follows:%
\begin{equation}
\left( \frac{2\left( \frac{t}{2}\right) ^{k}}{w\beta ^{b}e^{t}-a^{b}}\right)
^{h}=\sum_{n=0}^{\infty }y_{n,w,\beta }^{\left( h\right) }\left(
k,a,b\right) \frac{t^{n}}{n!},\text{ }\left\vert t+\log \left( w\left( \frac{%
\beta }{a}\right) ^{b}\right) \right\vert <2\pi .  \label{equation 28}
\end{equation}

By (\ref{equation 28}), we easily see that, 
\begin{eqnarray*}
\sum_{n=0}^{\infty }y_{n,w,\beta }^{\left( h\right) }\left( k,a,b\right) 
\frac{t^{n}}{n!} &=&2^{h}\left( \frac{t}{2}\right) ^{kh}\left( \frac{1}{%
w\beta ^{b}e^{t}-a^{b}}\right) ...\left( \frac{1}{w\beta ^{b}e^{t}-a^{b}}%
\right) \\
&=&2^{h}\left( \frac{t}{2}\right) ^{kh}\left( -1\right)
^{h}\sum_{n_{1}=0}^{\infty }w^{n_{1}}\left( \frac{\beta }{a}\right)
^{bn_{1}}e^{n_{1}t}...\sum_{n_{h}=0}^{\infty }w^{n_{h}}\left( \frac{\beta }{a%
}\right) ^{bn_{h}}e^{n_{h}t} \\
&=&2^{h}\left( \frac{t}{2}\right) ^{kh}\left( -1\right)
^{h}\sum_{n_{1},...,n_{h}=0}^{\infty }\phi _{w}\left( n_{1}+...+n_{h}\right)
\left( \frac{\beta }{a}\right) ^{b\left( n_{1}+...+n_{h}\right)
}e^{(n_{1}+...+n_{h})t}
\end{eqnarray*}

By using the taylor expansion of $e^{(n_{1}+...+n_{h})t}$ and by comparing
the coefficients of $t^{n}$ in the both side of the above equation, we
obtain that%
\begin{equation}
\frac{y_{n+kh,w,\beta }^{\left( h\right) }\left( k,a,b\right) }{(kh)!\binom{%
n+kh}{kh}}=2^{h\left( 1-k\right) }\left( -1\right) ^{h}\sum_{\underset{%
n_{1}+...+n_{h}\neq 0}{n_{1},...,n_{h}\geq 0}}^{\infty }\phi _{w}\left(
n_{1}+...+n_{h}\right) \left( \frac{\beta }{a}\right) ^{b\left(
n_{1}+...+n_{h}\right) }(n_{1}+...+n_{h})^{n}  \label{equation 29}
\end{equation}

From (\ref{equation 29}), we can define a unification of multiple twisted
zeta functions as follows:%
\begin{equation*}
\zeta _{\beta ,w}^{\left( h\right) }\left( s:k,a,b\right) =2^{h\left(
1-k\right) }\left( -1\right) ^{h}\sum_{\underset{n_{1}+...+n_{h}\neq 0}{%
n_{1},...,n_{h}=0}}^{\infty }\frac{\phi _{w}\left( n_{1}+...+n_{h}\right)
\left( \frac{\beta }{a}\right) ^{b\left( n_{1}+...+n_{h}\right) }}{%
(n_{1}+...+n_{h})^{s}}
\end{equation*}

for all $s\in 
\mathbb{C}
.$ We also obtain the following theorem in which a unification of multiple
twisted zeta functions interpolate a unification of \ multiple twisted Euler
and Genocchi polynomials at negative integer.

\begin{theorem}
For any $w\in T_{p},$ $n,h,k\in 
\mathbb{N}
$ $a,b\in 
\mathbb{R}
^{+};$ $\beta \in 
\mathbb{C}
,$ we obtain%
\begin{equation*}
\zeta _{\beta ,w}^{\left( h\right) }\left( -n:k,a,b\right) =\frac{%
y_{n+kh,w,\beta }^{\left( h\right) }\left( k,a,b\right) }{(kh)!\binom{n+kh}{%
kh}}.
\end{equation*}
\end{theorem}

\end{document}